\documentclass[11pt,a4paper]{article}

\usepackage{geometry}
\usepackage[T1]{fontenc}
\usepackage[utf8]{inputenc}
\usepackage{authblk}

\usepackage{graphics}
\usepackage{tikz,pgf}
\usetikzlibrary{decorations.markings, decorations.pathmorphing}
\usetikzlibrary{arrows,decorations,backgrounds,scopes,plotmarks}
\usepackage{subfigure}
\usepackage{here}
\usepackage{caption}
\usepackage{float}

\usepackage[english]{babel}
\usepackage[all]{xy}
\usepackage{marvosym}
\usepackage{stmaryrd,mathrsfs,amsmath,amssymb,bm,amsfonts}
\usepackage{enumerate}
\usepackage{amsmath,amsthm,amssymb,amscd}

\newtheorem{thm}{Theorem}
\newtheorem{defn}[thm]{Definition}

\newtheorem{ex}[thm]{Example}

\begin{document}

\title{A new definition of upward planar order}

\author[a]{Ting Li}

\author[b]{Xuexing Lu}

\affil[a,b]{\small School of Mathematical Sciences, Zaozhuang University, China}

\renewcommand\Authands{ and }
\maketitle

\begin{abstract}
We give a more coherent  definition of upward planar order.
\end{abstract}

\text{\textit{Keywords}: upward planar graph, upward planar order }\\

For a finite set $S$ with a linear order $<$ and any subset $X\subseteq S$, we write $X^-=\min{X}$ and $X^+=\max{X}$. The \textbf{convex hull} of $X$ in $S$ is the closed interval $\overline{X} = \{y\in S| X^-\leq y\leq X^+\}$.
In a directed graph, a vertex is called \textbf{processive} if it is neither a source nor a sink. We denote the sets of incoming edges and outgoing edges of a vertex $v$ by $I(v)$ and $O(v)$, respectively, and write $e_1\to e_2$ if there is a directed path starting from edge $e_1$ and ending with $e_2$. The set of incident edges of $v$ is denoted by $E(v)$, which is equal to $I(v)\sqcup O(v)$. The source and target vertices of an edge $e$ are denoted by $s(e)$ and $t(e)$, respectively. For an acyclic directed graph, the reachable order $\to$ is a partial order on its edge set.

\begin{defn}\label{upo}
An \textbf{upward planar order} on acyclic directed graph $G$ is a linear order $\prec$ on the edge set $E(G)$, such that

$(U1)$ $e_1\to e_2$ implies that $e_1\prec e_2$;

$(U2)$ for any vertex $v$, $\overline{I(v)}\cap \overline{O(v)}=\emptyset$ and $\overline{E(v)}=\overline{I(v)}\sqcup \overline{O(v)}$;

$(U3)$ for any two  vertices $v_1$ and $v_2$, $I(v_1)\cap \overline{I(v_2)}\neq\emptyset$ implies that $\overline{I(v_1)}\subseteq \overline{I(v_2)}$, and $O(v_1)\cap \overline{O(v_2)}\neq\emptyset$ implies that $\overline{O(v_1)}\subseteq \overline{O(v_2)}$.
\end{defn}
The $(U1)$ condition just says that $\prec$ is a linear extension of $\to$. For source vertices and sink vertices, the $(U2)$ condition is unconditionally true; for a processive vertex $v$, the $(U2)$ condition just means that $O(v)^-=I(v)^++1$.
Notice that for any two subsets $X, Y$ of a linearly ordered set, $X\subseteq \overline{Y}$ is equivalent to $\overline{X}\subseteq \overline{Y}$, thus the $(U3)$ condition can be restated as that for any two  vertices $v_1$ and $v_2$, $I(v_1)\cap \overline{I(v_2)}\neq\emptyset$ implies that $I(v_1)\subseteq \overline{I(v_2)}$, and $O(v_1)\cap \overline{O(v_2)}\neq\emptyset$ implies that $O(v_1)\subseteq \overline{O(v_2)}$.

Recall that a planar drawing of an acyclic directed graph is called \textbf{upward} if all edges increase monotonically in the vertical direction or any other fixed direction. An acyclic directed graph is called an \textbf{upward planar graph} if it admits an upward planar drawing. An acyclic directed graph together with an upward planar drawing is called an \textbf{upward plane graph}.
The notion of an upward planar order was introduced in \cite{[LY16]} to characterize upward planarity.
The main result of  \cite{[LY16]}, Theorem $6.1$, can be stated as follows.

\begin{thm}
An acyclic directed graph is upward planar  if and only if it admits an upward planar order.
\end{thm}

\begin{ex}
The following shows an example of an upward plane graph with its associated upward planar order, which is represented by the natural numbers labelling the edges.
\begin{center}
\begin{tikzpicture}[scale=0.27]

\node (v2) at (0.5,2) {};
\node (v1) at (-3,-3) {};
\node (v3) at (3,-2.5) {};
\node (v4) at (7,0) {};
\node (v8) at (0.5,-2) {};
\node (v5) at (-0.5,-7.5) {};
\node (v6) at (3.5,-5.5) {};
\node (v7) at (6,-8.5) {};
\draw  (-3,-3) -- (0.5,2)[postaction={decorate, decoration={markings,mark=at position .5 with {\arrowreversed[black]{stealth}}}}];
\draw  (3,-2.5) -- (0.5,2)[postaction={decorate, decoration={markings,mark=at position .5 with {\arrowreversed[black]{stealth}}}}];
\draw  (3,-2.5) -- (7,0)[postaction={decorate, decoration={markings,mark=at position .5 with {\arrowreversed[black]{stealth}}}}];
\draw  (-0.5,-7.5) -- (3.5,-5.5)[postaction={decorate, decoration={markings,mark=at position .5 with {\arrowreversed[black]{stealth}}}}];
\draw   (6,-8.5) -- (3.5,-5.5)[postaction={decorate, decoration={markings,mark=at position .5 with {\arrowreversed[black]{stealth}}}}];
\draw  (-0.5,-7.5) -- (-3,-3)[postaction={decorate, decoration={markings,mark=at position .5 with {\arrowreversed[black]{stealth}}}}];
\draw  (-0.5,-7.5) -- (0.5,-2)[postaction={decorate, decoration={markings,mark=at position .5 with {\arrowreversed[black]{stealth}}}}];
\draw  (3.5,-5.5)-- (0.5,-2)[postaction={decorate, decoration={markings,mark=at position .5 with {\arrowreversed[black]{stealth}}}}];
\draw   (6,-8.5) -- (7,0)[postaction={decorate, decoration={markings,mark=at position .5 with {\arrowreversed[black]{stealth}}}}];
\node (v9) at (2,-11) {};
\draw  (2,-11) --  (6,-8.5)[postaction={decorate, decoration={markings,mark=at position .5 with {\arrowreversed[black]{stealth}}}}];
\draw  (2,-11) -- (-0.5,-7.5)[postaction={decorate, decoration={markings,mark=at position .5 with {\arrowreversed[black]{stealth}}}}];
\draw  (3.5,-5.5) -- (7,0)[postaction={decorate, decoration={markings,mark=at position .5 with {\arrowreversed[black]{stealth}}}}];
\draw[fill] (v1) circle [radius=0.2];
\draw[fill] (v2) circle [radius=0.2];
\draw[fill] (v3) circle [radius=0.2];
\draw[fill] (v4) circle [radius=0.2];
\draw[fill] (v5) circle [radius=0.2];
\draw[fill] (v6) circle [radius=0.2];
\draw[fill] (v7) circle [radius=0.2];
\draw[fill] (v8) circle [radius=0.2];
\draw[fill] (v9) circle [radius=0.2];
\draw  plot[smooth, tension=.7] coordinates {(7,0) (8,-2.5) (8.5,-5.5) (7.5,-7.5) (6,-8.5)}[postaction={decorate, decoration={markings,mark=at position .5 with {\arrow[black]{stealth}}}}];
\node [scale=0.68]at (-1.9,0) {$1$};
\node [scale=0.68]at (-2.5,-5.6) {$2$};
\node [scale=0.68]at (-0.5,-4.3) {$3$};
\node [scale=0.68]at (2.1,-2.9) {$4$};
\node [scale=0.68]at (2.2,0.5) {$5$};
\node [scale=0.68]at (4.8,-0.7) {$6$};
\node [scale=0.68]at (4.1,-3.7) {$7$};
\node [scale=0.68]at (1.9,-7.2) {$8$};
\node [scale=0.68]at (0.1,-9.7) {$9$};
\node [scale=0.68]at (4.2,-7.4) {$10$};
\node[scale=0.68] at (5.7,-5) {$11$};
\node [scale=0.68]at (9.4,-4.5) {$12$};
\node [scale=0.68]at (4.4,-10.5) {$13$};

\node (v12) at (14,0.25) {};
\node (v13) at (13,-11) {};
\draw[fill] (v12) circle [radius=0.2];
\draw[fill] (v13) circle [radius=0.2];
\node (v14) at (11.25,-5.25) {};
\node (v15) at (15.25,-6.5) {};
\draw (14,0.25)  -- (11.25,-5.25)[postaction={decorate, decoration={markings,mark=at position .5 with {\arrow[black]{stealth}}}}];
\draw (11.25,-5.25) -- (13,-11)[postaction={decorate, decoration={markings,mark=at position .5 with {\arrow[black]{stealth}}}}];
\draw (14,0.25) -- (15.25,-6.5)[postaction={decorate, decoration={markings,mark=at position .5 with {\arrow[black]{stealth}}}}];
\draw (15.25,-6.5)-- (13,-11)[postaction={decorate, decoration={markings,mark=at position .5 with {\arrow[black]{stealth}}}}];
\node (v16) at (13.25,-3) {};
\node (v17) at (12.75,-8) {};
\draw  (13.25,-3) --(12.75,-8)[postaction={decorate, decoration={markings,mark=at position .5 with {\arrow[black]{stealth}}}}];
\draw[fill] (v14) circle [radius=0.2];
\draw[fill] (v15) circle [radius=0.2];
\draw[fill] (v16) circle [radius=0.2];
\draw[fill] (v17) circle [radius=0.2];
\node[scale=0.68] at (12,-2) {$14$};
\node[scale=0.68] at (11.25,-8.75) {$15$};
\node[scale=0.68] at (13.75,-5.75) {$16$};
\node[scale=0.68] at (15.25,-3) {$17$};
\node [scale=0.68] at (14.75,-9.25) {$19$};
\node (v18) at (17.75,-3.5) {};
\node (v19) at (17,-9) {};
\draw (17.75,-3.5)-- (17,-9)[postaction={decorate, decoration={markings,mark=at position .5 with {\arrow[black]{stealth}}}}];
\draw[fill] (v18) circle [radius=0.2];
\draw[fill] (v19) circle [radius=0.2];
\node [scale=0.68] at (18.25,-6.5) {$20$};
\draw  plot[smooth, tension=.7] coordinates {(14,0.25) (16.25,-2.25)  (15.25,-6.5)}[postaction={decorate, decoration={markings,mark=at position .5 with {\arrow[black]{stealth}}}}];
\node[scale=0.68] at (16.75,-1) {$18$};
\end{tikzpicture}
\end{center}
\end{ex}

\begin{ex}\label{ex2}
Upward planar orders encode the information of how to draw acyclic directed graphs on the plane.
The following figure shows two different upward planar orders  associated to two different upward planar drawings of the acyclic directed graph.
$$
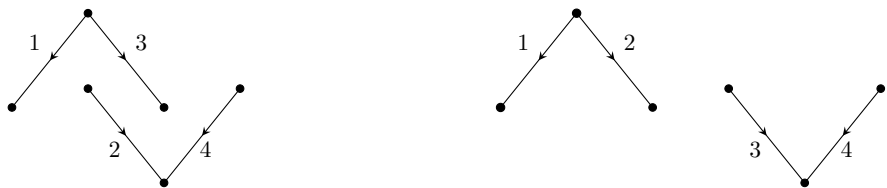
\begin{matrix}
\begin{matrix}
\begin{tikzpicture}[scale=0.5]
\node (v1) at (-0.5,2.5) {};
\draw[fill] (-0.5,2.5) circle [radius=0.1];
\node (v2) at (-2.5,0) {};
\draw[fill](-2.5,0) circle [radius=0.1];
\node (v3) at (1.5,0) {};
\draw[fill] (1.5,0) circle [radius=0.1];
\node (v4) at (-0.5,0.5) {};
\draw[fill] (-0.5,0.5) circle [radius=0.1];
\node (v5) at (1.5,-2) {};
\draw[fill] (1.5,-2) circle [radius=0.1];
\node (v6) at (3.5,0.5) {};
\draw[fill] (3.5,0.5) circle [radius=0.1];
\draw  (-0.5,2.5) -- (-2.5,0)[postaction={decorate, decoration={markings,mark=at position .5 with {\arrow[black]{stealth}}}}];
\node [scale=0.8][above] at (-1.9,1.3) {$1$};
\draw  (-0.5,2.5) -- (1.5,0)[postaction={decorate, decoration={markings,mark=at position .5 with {\arrow[black]{stealth}}}}];
\node [scale=0.8][above] at (0.9,1.3) {$3$};
\draw  (-0.5,0.5) --  (1.5,-2)[postaction={decorate, decoration={markings,mark=at position .5 with {\arrow[black]{stealth}}}}];
\node[scale=0.8] [below] at (0.2,-0.7) {$2$};
\draw  (3.5,0.5) --  (1.5,-2)[postaction={decorate, decoration={markings,mark=at position .5 with {\arrow[black]{stealth}}}}];
\node[scale=0.8] [below] at (2.6,-0.7) {$4$};
\end{tikzpicture}
\end{matrix}&&&&&&&&&
\begin{matrix}
\begin{tikzpicture}[scale=0.5]

\node (v1) at (-0.5,2.5) {};
\draw[fill] (-0.5,2.5) circle [radius=0.11];
\node (v2) at (-2.5,0) {};
\draw[fill](-2.5,0) circle [radius=0.11];

\node (v3) at (1.5,0) {};
\draw[fill] (1.5,0) circle [radius=0.1];
\node (v4) at (-0.5,0.5) {};
\draw[fill] (3.5,0.5) circle [radius=0.1];
\node (v5) at (5.5,-2) {};
\draw[fill] (5.5,-2) circle [radius=0.1];
\node (v6) at (7.5,0.5) {};
\draw[fill] (7.5,0.5) circle [radius=0.1];
\draw  (-0.5,2.5) -- (-2.5,0)[postaction={decorate, decoration={markings,mark=at position .5 with {\arrow[black]{stealth}}}}];
\node[scale=0.8] [above] at (-1.9,1.3) {$1$};
\draw  (-0.5,2.5) -- (1.5,0)[postaction={decorate, decoration={markings,mark=at position .5 with {\arrow[black]{stealth}}}}];
\node [scale=0.8][above] at (0.9,1.3) {$2$};
\draw  (3.5,0.5) --  (5.5,-2)[postaction={decorate, decoration={markings,mark=at position .5 with {\arrow[black]{stealth}}}}];
\node[scale=0.8] [below] at (4.2,-0.7) {$3$};
\draw  (7.5,0.5) --  (5.5,-2)[postaction={decorate, decoration={markings,mark=at position .5 with {\arrow[black]{stealth}}}}];
\node[scale=0.8] [below] at (6.6,-0.7) {$4$};
\end{tikzpicture}
\end{matrix}
\end{matrix}
$$
\end{ex}

Now we give a more coherent definition as follows.
\begin{defn}\label{upo1}
An \textbf{upward planar order} on acyclic directed graph $G$ is a linear order $\prec$ on the edge set $E(G)$, such that

$(Q1)$ $e_1\rightarrow e_2$ implies that $e_1\prec e_2$;

$(Q2)$ if $e_1\prec e\prec e_2$ and $e_1$, $e_2$ are adjacent, then
\begin{equation*}
\begin{cases}
I(t(e))\subseteq \overline{I(v)},&  \text{if}\ t(e_1)=t(e_2)=v;\\
O(s(e))\subseteq \overline{O(v)},&  \text{if}\ s(e_1)=s(e_2)=v;\\
\text{either}\ I(t(e))\subseteq \overline{I(v)}\  \text{or}\ O(s(e))\subseteq \overline{O(v)},&  \text{if}\ t(e_1)=s(e_2)=v.
\end{cases}
\end{equation*}
\end{defn}

The four possible configurations of $e_1\prec e\prec e_2$ with $e_1, e_2$ adjacent are shown as follows.
\begin{center}
\begin{tikzpicture}[scale=0.5]
\node (v2) at (-4.5,1.5) {};
\node (v1) at (-6,3.5) {};
\node (v3) at (-3,3.5) {};
\draw  (v1) -- (-4.5,1.5)[postaction={decorate, decoration={markings,mark=at position .5 with {\arrow[black]{stealth}}}}];
\draw  (v3) --(-4.5,1.5)[postaction={decorate, decoration={markings,mark=at position .5 with {\arrow[black]{stealth}}}}];
\node (v5) at (-4.5,2.5) {};
\node (v4) at (-4.5,4.5) {};
\draw (-4.5,4.5)-- (-4.5,2.5)[postaction={decorate, decoration={markings,mark=at position .5 with {\arrow[black]{stealth}}}}];
\draw[fill] (v2) circle [radius=0.1];
\draw[fill] (v5) circle [radius=0.1];
\node at (-6,2.5) {$e_1$};
\node at (-3,2.5) {$e_2$};
\node at (-4,4) {$e$};
\node (v7) at (3,2.5) {};
\node (v6) at (2,5) {};
\node (v8) at (4,0) {};
\draw  (2,5) --(3,2.5)[postaction={decorate, decoration={markings,mark=at position .5 with {\arrow[black]{stealth}}}}];
\draw  (3,2.5)-- (4,0)[postaction={decorate, decoration={markings,mark=at position .5 with {\arrow[black]{stealth}}}}];
\node (v9) at (4.5,5) {};
\draw  (4.5,5)-- (3,2.5)[postaction={decorate, decoration={markings,mark=at position .5 with {\arrow[black]{stealth}}}}];
\draw[fill] (v7) circle [radius=0.1];
\node (v11) at (3,3.5) {};
\node (v10) at (3.5,5.5) {};
\draw  (3.5,5.5) --(3,3.5)[postaction={decorate, decoration={markings,mark=at position .5 with {\arrow[black]{stealth}}}}];
\draw[fill] (v11) circle [radius=0.1];
\node (v13) at (9.5,2.5) {};
\node (v12) at (8,5) {};
\node (v14) at (10.5,0) {};
\draw  (8,5) -- (9.5,2.5)[postaction={decorate, decoration={markings,mark=at position .5 with {\arrow[black]{stealth}}}}];
\draw  (9.5,2.5) -- (10.5,0)[postaction={decorate, decoration={markings,mark=at position .5 with {\arrow[black]{stealth}}}}];
\draw[fill] (v13) circle [radius=0.11];
\node (v15) at (8,0) {};
\draw  (9.5,2.5) -- (8,0)[postaction={decorate, decoration={markings,mark=at position .5 with {\arrow[black]{stealth}}}}];
\node (v16) at (9.5,1.5) {};
\node (v17) at (9,-1) {};
\draw  (9.5,1.5)-- (9,-0.5)[postaction={decorate, decoration={markings,mark=at position .5 with {\arrow[black]{stealth}}}}];
\draw[fill] (v16) circle [radius=0.1];
\node (v18) at (16,3) {};
\node (v19) at (14.5,0.5) {};
\node (v20) at (17.5,0.5) {};
\node (v21) at (16,2) {};
\node (v22) at (16,-1) {};
\draw  (16,3) -- (14.5,0.5)[postaction={decorate, decoration={markings,mark=at position .5 with {\arrow[black]{stealth}}}}];
\draw (16,3) --(17.5,0.5)[postaction={decorate, decoration={markings,mark=at position .5 with {\arrow[black]{stealth}}}}];
\draw  (16,2) -- (16,-0.5)[postaction={decorate, decoration={markings,mark=at position .5 with {\arrow[black]{stealth}}}}];
\draw[fill] (v21) circle [radius=0.1];
\draw[fill] (v18) circle [radius=0.1];
\node[scale=0.8] at (2,4) {$e_1$};
\node[scale=0.8] at (4.1,1) {$e_2$};
\node[scale=0.8] at (3,5) {$e$};
\node[scale=0.8] at (8,4) {$e_1$};
\node[scale=0.8] at (10.6,1) {$e_2$};
\node[scale=0.8] at (9.5,-0.5) {$e$};
\node[scale=0.8] at (14.5,1.5) {$e_1$};
\node[scale=0.8] at (17.5,1.5) {$e_2$};
\node[scale=0.8] at (16.5,0.5) {$e$};
\node[scale=0.8] at (2.5,2.2) {$v$};
\node [scale=0.8]at (10,2.6) {$v$};
\node[scale=0.8] at (16,3.5) {$v$};
\node [scale=0.8]at (-4.5,1) {$v$};
\node[scale=0.8] at (-4.5,-2) {$I(t(e))\subseteq \overline{I(v)}$};
\node [scale=0.8]at (3,-2) {$I(t(e))\subseteq \overline{I(v)}$};
\node [scale=0.8]at (9.5,-2) {$O(s(e))\subseteq \overline{O(v)}$};
\node [scale=0.8]at (16,-2) {$O(s(e))\subseteq \overline{O(v)}$};
\end{tikzpicture}
\end{center}

\begin{thm}\label{eq}
Definition \ref{upo} and Definition \ref{upo1} are equivalent.
\end{thm}
\begin{proof}
Clearly, $(U1)=(Q1)$. We only need to show that $(U2)+(U3)\Longleftrightarrow (Q2)$.

$(\Longleftarrow)$. To show $(Q2)\Longrightarrow(U2)$, we only need to show that for any processive vertex $v$, $O(v)^-=I(v)^++1$. We prove this by contradiction. Suppose there exist a processive vertex $v$ and an edge $e$ such that $I(v)^+\prec e \prec O(v)^-$. Clearly, $v=t(I(v)^+)=s(O(v)^-)$, then by $(Q2)$, we have either $e\in I(t(e))\subseteq \overline{I(v)}$ or $e\in O(s(e))\subseteq\overline{O(v)}$, both of which  contradict $I(v)^+\prec e \prec O(v)^-$.

Now we show $(Q2)\Longrightarrow(U3)$. Assume $I(v_1)\cap \overline{I(v_2)}\neq \emptyset$, then there exists $e\in I(v_1)$ such that $I(v_2)^-\prec e \prec I(v_2)^+$. Clearly, $v_2=t(I(v_2)^-)=t(I(v_2)^+)$, then by $(Q2)$, we have $I(v_1)=I(t(e))\subseteq \overline{I(v_2)}$. Similarly, $O(v_1)\cap \overline{O(v_2)}\neq \emptyset$ means that there exists $e\in O(v_1)$ such that $O(v_2)^-\prec e \prec O(v_2)^+$. By $v_2=s(O(v_2)^-)=s(O(v_2)^+)$ and $(Q2)$, we have $O(v_1)=O(s(e))\subseteq \overline{O(v_2)}$.

$(\Longrightarrow)$. Assume $e_1\prec e\prec e_2$ and $e_1$, $e_2$ are adjacent, then there are three cases.

\noindent\textbf{Case 1:} $t(e_1)=t(e_2)=v$. In this case, $e\in I(t(e))\cap\overline{I(v)}$, by $(U3)$, we have $I(t(e))\subseteq\overline{I(v)}$.

\noindent\textbf{Case 2:} $s(e_1)=s(e_2)=v$. In this case, $e\in O(s(e))\cap\overline{O(v)}$, by $(U3)$, we have $I(s(e))\subseteq\overline{O(v)}$.

\noindent\textbf{Case 3:} $t(e_1)=s(e_2)=v$. In this case, $e\in \overline{E(v)}$, then by $(U2)$, we have either $e\in \overline{I(v)}$ or $e\in \overline{O(v)}$, which reduces the proof to case 1 or case 2, respectively.
\end{proof}


\textbf{Ting  Li}\hfill \\  Email:liting810505@163.com

\textbf{Xuexing Lu}\hfill \\  Email: xxlu@uzz.edu.cn

\end{document}